\documentclass[a4paper,12pt]{amsart}
\usepackage{mycommands}

\usepackage{amsfonts,amssymb,amsmath,times} 
\usepackage{url}

\makeatletter
\def\url@leostyle{%
  \@ifundefined{selectfont}{\def\UrlFont{\sf}}{\def\UrlFont{\small\ttfamily}}}
\makeatother

\def\abar{\bar{a}}
\def\bbar{\bar{b}}
\def\cbar{\bar{c}}

\def\hbar{\bar{h}}

\def\vbar{\bar{v}}

\def\xbar{\bar{x}}
\def\ybar{\bar{y}}

\def\lg{{\rm lg}}
\def\phi{\varphi}

\def\QQ{\mathbb{Q}}
\def\ZZ{\mathbb{Z}}

\newcommand{\seq}[2][\omega]{\langle{#2}_i\rangle_{i\in {#1}}}

\newcommand{\bbd}[1]{\mathbb{#1}}

\title{dp-rank and forbidden configurations}
\author{Hunter Johnson}\thanks{Dept. Math \& CS, John Jay College, CUNY, 444 W. 59th St., New York, NY 10019.}
\thanks{hujohnson@jjay.cuny.edu.} 

\begin{document}

\begin{abstract}
 A theory $T$ is shown to have an ICT pattern of depth $k$ in $n$ variables iff it interprets some $k$-maximum VC class in $n$ parameters.
\end{abstract}

\maketitle

\section{Introduction} \label{S:S1}

We begin with the definition of an ICT pattern.  Our definition comes from Adler \cite{Ad07} as adapted from Shelah.  The definition assumes an ambient theory $T$ in a language $L$, and a monster model from which the parameters are taken. 

\begin{definition}
 For an cardinal $\kappa$, an \textit{ICT pattern} of depth $\kappa$ in variables $\xbar$ is a set of formulas $\{\psi_\alpha(\xbar;\ybar_\alpha): \alpha < \kappa\}$ together with an array  $\langle \bbar_n^\alpha: \alpha < \kappa, n < \omega \rangle$ such that $\lg(\bbar_n^\alpha) = \lg(\ybar_\alpha)$ and for any $\eta: \kappa \rightarrow \omega$, the set of formulas
\begin{equation}
 \{\psi_\alpha(\xbar;\bbar_{\eta(\alpha)}^\alpha):\alpha <\kappa\} \cup \{\neg \psi_\alpha(\xbar;\bbar_n^\alpha):\alpha <\kappa, n < \omega,\eta(\alpha)\neq n\}
\end{equation}
is consistent.
\end{definition}

 Intuitively an ICT pattern constitutes an array of formulas with $\kappa$ rows and $\omega$ columns, such that for any ``path'' downward through the array it is consistent that exactly the formulas appearing on the path are non-negated.  The acronym stands for \textit{independent contradictory types}.

Though ICT patterns and definitions of other similarly array-based notions (such as INP patterns) appear in Shelah, interest in them partly stems from Onshuus and Usvyatsov \cite{OnUs11}, who extracted from Shelah a simple concept of dp-rank, and in particular dp-minimality.

Shelah investigated a cardinal invariant of a theory $T$, denoted $\kappa_{ict}$, defined as the least infinite cardinal (should it exist) such that $T$ does not admit an ICT pattern of depth $\kappa_{ict}$.  When $\kappa_{ict}$ exists, $T$ is said to be \textit{dependent}, and when $\kappa_{ict} = \aleph_0$, $T$ is said to be \textit{strongly dependent} \cite{Sh07,Sh06}.  Note that because many formulas are involved in the definition of an ICT pattern, strong dependence does not imply a finite bound on the possible depth of an ICT pattern. Nonetheless it is possible to study the properties of finite cardinal bounds as well.  Onshuus and Usvyatsov define \textit{dp-rank} for a partial type $p(\xbar)$ as the maximum cardinal $\kappa$ (possibly finite) such that $p(\xbar)$ is consistent with an ICT pattern in variables $\xbar$ of depth $\kappa$. 

In this paper we generally assume $p(\xbar) = \{\xbar = \xbar\}$, so that we are really considering the dp-rank of a certain sequence of variables $\xbar$.  We define $dpR_T(n)$ as the dp-rank of any partial type $\{\xbar = \xbar\}$ in $T$, where $lg(\xbar)=n$, and all variable symbols occuring in $\xbar$ are distinct.  This is clearly independent of the particular $\xbar$ chosen.  When the theory under consideration is obvious we suppress the dependence on $T$ and simply write $dpR(n)$.

As is frequently the case in model theoretic definitions, dp-rank can be understood in terms of the interpretability of certain set systems in models of the theory.  Another view of dependence for a theory $T$, for example, is that every model of the theory is forbidden from interpreting the power-set of an infinite set.  Stability is well-known to be equivalent to every model of the theory being forbidden from interpreting an infinite chain of sets.  In both cases the interpretation must be uniform; we give a precise description of our notion of interpretation at the end of Section \ref{S:DF}. While these classical concepts are defined on the basis of finite/infinite distinctions, we wish to consider the more fine-grained question of which set families can be interpreted in $M^n$ when $\mcm \models T$ and $dpR_T(n) = k$.

In this paper (see Theorem \ref{T:T2}) we have shown that a cardinality-based property of interpretable set systems is sufficient to characterize $dpR_T(n)$.  The property we consider, the maximum property, can be viewed as a homogeneity condition on VC density (this notion is described in \cite{AsDoHaMaSt11}).  Along the course of our investigation we encounter set systems which are characterized in terms of certain forbidden configurations (see Definition \ref{D:bigD}).  We describe the relation of these forbidden configurations to the alternation properties of a dependent formula, and to dp-rank.

\section{Definitions and basic facts}\label{S:DF}

In this section we introduce notation and give some background on VC classes.  For the purposes of the paper, fix a complete theory $T$ in a language $L$. We consider $L$ formulas $\phi(\xbar;\ybar)$ which are \textit{partitioned} in the sense that the $\ybar$ variables are viewed as parameters.  The semicolon indicates the separation of variables.  We use the symbol $\mcm$ to denote a monster model of $T$. The model $\mcm$ is assumed to be saturated in a high cardinality, and to be sufficiently large to admit an elementary embedding of all other models and sets considered.  We will be interested in combinatorial properties of formulas $\phi(\xbar;\ybar)$. These are sometimes conveniently expressed by considering the family of sets defined by $\phi$ on $\mcm$ as its parameters vary. 

We use the convention that whenever $A \sse M^{|\xbar|}$ and $\bbar \in M^{|\ybar|}$, the symbol $\phi(A;\bbar)$ denotes
$$\phi(A;\bbar) = \{\abar \in A: \mcm \models \phi(\abar;\bbar)\}$$

For $A \sse M^{|\ybar|}$ and $B \sse M^{|\xbar|}$, let $\mathcal{C}_\phi(B)^A=\{\phi(B,\bbar): \bbar \in A\}$. We let $\mathcal{C}_\phi(B)$ where no parameter set is specified implicitly denote $\mathcal{C}_\phi(B)^{M^{|\ybar|}}$. The abbreviation $\mathcal{C}_\phi(\mcm)$ will be used for $\mathcal{C}_\phi(M^{|\xbar|})^{M^{|\ybar|}}$. 

It was observed by Laskowski \cite{L92} that the independence dimension of $\phi(\xbar;\ybar)$ is the Vapnik-Chervonenkis (VC) dimension of $\mathcal{C}_\phi(\mcm)$. We give several definitions related to VC dimension. For a set $X$, we represent the power-set of $X$ by $2^X = \{A:A\sse X\}$.

\begin{definition}
 Let $X$ a set, $A \sse X$, and $\mathcal{C} \sse 2^X$.  Define $\mathcal{C}(A) = \{c\cap A: c\in \mathcal{C}\}$.  Say that $\mathcal{C}$ \textit{shatters} $A$ if $\mathcal{C}(A)=2^A$.  Let the VC dimension of $\mathcal{C}$, denoted VC$(\mathcal{C})$, be defined as $\sup\{|A|: A \sse X, \mathcal{C} \text{ shatters } A\}$.  We say that $\gc$ is a \textit{VC class} if VC($\gc) < \omega$.
\end{definition}

It is clear from Laskowski's observation that $T$ is dependent (or NIP) if and only if every $L$ formula $\phi(\xbar;\ybar)$ induces a definable family in $\mcm$ which is a VC class.

For $n \in \omega$, $d \in \omega$, define $\Phi_d(n) = \sum_{i=0}^d {\binom{n}{i}}$ if $n \geq d$ and $2^n$ otherwise.

The following lemma was discovered independently by Sauer, Perles and Shelah, and in an asymptotic form by Vapnik and Chervonenkis.

\begin{lemma}[Sauer's Lemma \cite{Sa72,Sh72,VaCh71}]
 Suppose $\mathcal{C} \sse 2^X$ for a set $X$. If VC$(\mathcal{C}) = d$, and $A \sse X$ is finite, then $$|\mathcal{C}(A)| \leq \Phi_d(|A|)$$
\end{lemma}

We now define maximum VC classes, which were investigated first by Welzl \cite{We87}, who called them complete range spaces. They are defined by the property that they always realize the bound imposed by Sauer's Lemma.  

\begin{definition}
 Suppose $\mathcal{C} \sse 2^X$ and VC$(\mathcal{C})=d$.  Say that $\mathcal{C}$ is \textit{maximum} of VC-dimension $d$ (or $d$-maximum) if for all finite $A \sse X$, $$|\mathcal{C}(A)| = \Phi_d(|A|)$$
\end{definition}

If sets are added to a VC class until no more can be added without increasing the VC dimension, the result is not necessarily maximum.  Therefore the following definition is useful.

\begin{definition}[\cite{Du99}]
 Suppose $\mathcal{C} \sse 2^X$ and VC$(\mathcal{C})=d$.  Say that $\mathcal{C}$ is \textit{maximal} of VC-dimension $d$ (or $d$-maximal) if for any $c \in 2^X \setminus \mathcal{C}$, VC$(\mathcal{C} \cup \{c\}) = d+1$.
\end{definition}

\begin{proposition}\label{P:P991}
If for a finite set $X$, $\gc \sse 2^X$ is maximum, then it is also maximal.
\end{proposition}
 \begin{proof}
  This follows from Sauer's lemma.
 \end{proof}

\begin{definition}
 Say that a partitioned formula $\phi(\xbar;\ybar)$ is $d$-\textit{maximum} (\textit{maximal}) in $\mcm$ if $\mathcal{C}_\phi(\mcm)$ is $d$-maximum (maximal).  
\end{definition}

While being maximum does not depend on the model used in the above definition, being maximal does.

  Let $\mathcal{C} \sse 2^X$ be $d$-maximum.  For any $A \sse X$ with $|A|=d+1$, $|\mathcal{C}(A)|=\Phi_d(d+1)=2^{d+1}-1$.  Let the unique $A^* \in 2^A \setminus \mathcal{C}(A)$ be called the \textit{forbidden label} for $\mathcal{C}$ on $A$ (Floyd's thesis \cite{Fl89}, Section 3.4). 

\begin{example}
 Let $X$ an infinite set, $d \in \omega$ and $\mathcal{C} = [X]^d$.  Then for any $A \sse X$ of cardinality $d+1$, the forbidden label for $\mathcal{C}$ on $A$ is $A$ itself.
\end{example}

\begin{example}\label{E:E33}
 Let $X = \bbd{Q}$ and $\mathcal{C} = \mathcal{C}_{x<y}(\bbd{Q})$.  Then for $\{a,b\} \sse \bbd{Q}$ with $a < b$, the forbidden label for $\mathcal{C}$ on $\{a,b\}$ is $\{b\}$.
\end{example}

  For a set $X$ and $n \in \omega$, we use the notation $[X]^n = \{A \sse X: |A|=n\}$ and $[X]^{\leq n} = \{A \sse X: |A|\leq n\}$.
\begin{definition} \label{D:bigD}
 Let $X$ be a set linearly ordered by $<$.  Let $\sigma:[X]^{d+1} \rightarrow 2^{d+1}$ assign a forbidden label to every subset of $X$ of size $d+1$ by associating every $A = \{a_0,\ldots,a_{d}\}$, $a_0 < a_1 < \cdots < a_{d}$ with the forbidden label $A_\sigma = \{a_i \in A : \sigma(A)(i) = 1\}$.  Say that $\gc \sse 2^X$ is \textit{characterized} by $\sigma$ if, for all $c \sse X$, $c \in \gc \iff \forall A \in [X]^{d+1}(c \cap A \neq A_\sigma)$.
\end{definition}

If $\sigma$ is constantly $\eta$ for some $\eta \in 2^{d+1}$ and $\gc$ is characterized by $\sigma$, we will abuse notation and say that $\gc$ is characterized by $\eta$. Sometimes we will refer to $\eta$ as a forbidden label, even though it is technically only a bit string. We do this because $\eta$ gives the form for all forbidden labels in $\gc$. 

We will usually be interested in finite sets, and so the following definition is convenient.
  
\begin{definition} 
For a linearly ordered set $(X,<)$, $\gc \sse 2^X$, $d \in \omega$, and $\eta \in 2^{d+1}$, say that $\gc$ is \textit{finitely characterized} by $\eta$ if for every finite $X_0 \sse X$, $\gc(X_0)$ is characterized by $\eta$.
\end{definition}

We would like to establish that if $\gc\sse 2^X$ is characterized by $\eta \in 2^{d+1}$, then $\gc$ is finitely characterized by $\eta$. Toward this end we give the following lemma.  We say that $c \in \gc$ \textit{traces} (or \textit{induces}) $\eta$ on $B \sse X$ if there are $b_0 <\ldots<b_{d}$ in $B$ such that $b_i \in c$ iff $\eta(i)=1$. 

\begin{lemma} \label{L:978}
 Let $(X,<)$ be a linearly ordered set, and choose any finite $B \sse X$, $d \in \omega$, and $\eta \in 2^{d+1}$.  Then for any $c \sse B$ not inducing $\eta$ on $B$ there can be found some $c' \sse X$ which does not induce $\eta$ on $X$ and such that $c' \cap B = c \cap B$.
\end{lemma}
\begin{proof}
 We prove the lemma by induction on $\eta$ as a binary string.  The base cases $\eta = \langle 0 \rangle$ and $\eta = \langle 1 \rangle$ are clear.  Now suppose for $s,t \in 2$, $\mu \in 2^{d+1}$ has ending digit $s$, and $\eta = \mu^\frown\langle t \rangle$.  

Let a finite $B \sse X$ be given.  Suppose $c \sse B$ does not induce $\eta$ on $B$. If $c$ does not induce $\mu$ on $B$, then by inductive hypothesis there exists $c' \sse X$ such that $c'$ does not induce $\mu$ on $X$.  A fortiori $c'$ fails to trace $\eta$ as well.  

Now suppose $c$ does induce $\mu$ on $B$. Let $b_0 <\cdots < b_d$ be a least witness in the sense that $b_d$ is lowest. Define $B_{<b_d} = \{b \in B: b< b_d\}$. Now $c \cap B_{<b_d}$ does not induce $\mu$ on $B_{<b_d}$, and so, by inductive hypothesis, there is $c' \sse X$ such that $c' \cap B_{<b_d} = c \cap B_{<b_d}$ and $c'$ does not induce $\mu$ on $X$.  Let $\chi_{c'}:X \rightarrow 2$ be the characteristic function of $c'$, with $\chi_{c'}(a)=1$ iff  $a \in c$. Define $\chi_{c}$ similarly.  Then $\chi_{c}(b_d)=s$, and $\chi_{c}$ is constantly $1-t$ on $B_{> b_d}$, for otherwise $c$ would induce $\eta$.  Define
\begin{equation}
\chi^*(x) = \begin{cases}
  \chi_{c'}(x) &\text{ if } x < b_d \\ s &\text{ if } x = b_d \\ 1-t &\text{ if } x > b_d
\end{cases}
\end{equation}
Then $\chi^*$ is a total function on $X$ which agrees with $\chi_{c}$ on $B$.  Let $c^*$ be the set associated to the characteristic function $\chi^*$.

We must show that $c^*$ does not induce $\eta$ on $X$.  By way of contradiction, suppose there are $a_0<\ldots<a_{d+1}$ such that $a_i \in c^*$ iff $\eta(i)=1$.  By choice of $c'$, $a_d \geq b_d$. Therefore $a_{d+1} > b_d$.  Then $\chi^*(a_{d+1})=1-t$ by definition of $\chi^*$, and $\chi^*(a_{d+1})=t$ by definition of $\eta$.  This is a contradiction.  

\end{proof}
Note that in Lemma \ref{L:978}, the assumption that $B$ is finite can be removed if $(X,<)$ is a well-ordering, since that assumption is only used to get a least witness.

\begin{corollary}\label{T:T55}
 For any linearly ordered set $(X,<)$, $\eta \in 2^{d+1}$, and $\gc \sse 2^X$, if $\gc$ is characterized by $\eta$, then $\gc$ is finitely characterized by $\eta$.
\end{corollary}
\begin{proof}
 Let $B \sse X$ be a finite subset.  Clearly nothing in $\gc(B)$ traces $\eta$ on $B$. Suppose $c\sse B$ does not trace $\eta$.  By Lemma \ref{L:978} there is $c' \sse X$ which does not trace $\eta$ on $X$ and such that $c' \cap B = c$.  Then by hypothesis $c' \in \gc$, and therefore $c \in \gc(B)$.
\end{proof}

When $(X,<)$ is well-ordered the above corollary can be strengthened to say that if $\gc$ is characterized by $\eta$, then $\gc(X_0)$ is characterized by $\eta$ for any $X_0 \sse X$.

The following definitions will be needed in the next section.

\begin{definition}
If $(I,<)$ is a linear order and $\seq[I]{\abar}$ is a sequence of points in $M^n$, we say the sequence is \textit{indiscernible}\textit{} if for every formula $\phi(\vbar_1,\ldots,\vbar_n)$ and subsequences $i_1 < \cdots < i_n$ and $j_1 < \cdots < j_n$ of $I$, $\mcm \models \phi(\abar_{i_1},\ldots,\abar_{i_n}) \equiv \phi(\abar_{j_1},\ldots,\abar_{j_n})$.
\end{definition}

We will classify maximum VC classes on indiscernible sequences modulo the following equivalence relation, which we call \textit{similarity}.  

\begin{definition}
 If $\gc_1 \sse 2^X$ and $\gc_2 \sse 2^X$, say $\gc_1 \sim \gc_2$ if for every finite $A \sse X$, $\gc_1(A) = \gc_2(A)$.  
\end{definition}

Note that if $\gc_1 \sim \gc_2$ and $\gc_1$ is finitely characterized by some $\eta \in 2^{d+1}$, then $\gc_2$ is also finitely characterized by $\eta$. Also, any $\gc_1$ and $\gc_2$ finitely characterized by the same $\eta \in 2^{d+1}$ will have $\gc_1 \sim \gc_2$.

Say that the theory $T$ \textit{interprets} $\gc \sse 2^X$ in $n$ parameters if there is a $L$-formula $\phi(\xbar;\ybar)$, $\ybar = \langle y_1,\ldots,y_n\rangle$, and an injection $f:X \rightarrow M^{|\xbar|}$ such that for all $c \in \gc$ there is $\bbar_c \in M^n$ such that $$f(c) = \phi(f(X);\bbar_c) = \{\abar \in f(X): \mcm \models \phi(\abar;\bbar_c)\}$$

Note that there could exist extraneous $\bbar$ so that $\phi(f(X);\bbar) \notin \{f(c):c \in \gc\}$.

\section{Alternation conditions and forbidden labels}

Set systems $\gc \sse 2^X$ which are characterized by some $\eta \in 2^{d+1}$ for $d \in \omega$ will play a central role in the results below, and therefore we will say a few words about how these can be understood. We offer no proofs in this section, though the claims can be easily derived by considering the proof of Lemma \ref{C:C1} (see the remark following that lemma).

  When $\gc \sse 2^X$ is characterized by $\eta$, with the implicit ordering on $X$, every set in $\gc$ is given a geometric form by $\eta$ in a way that is similar to, in fact stronger than, the restrictions given by alternation number. Adler \cite{Ad10} includes a discussion of alternation number, which is usually defined on an indiscernible sequence. 

\begin{definition}
For a linear order $(X,<)$ and $A\sse X$, the\textit{ alternation number} of $A$ in $X$ is $n \in \omega$ if there are $a_1<\cdots<a_n \in X$ such that $a_i \in A$ iff $a_{i+1} \notin A$ for all $i=1,\ldots,n$, and there are not $n+1$ such elements in $X$. 
\end{definition}

The alternation number of a family $\gc$ is naturally defined as the supremum of the alternation numbers of its member sets.  Note that any $\gc$ finitely characterized by a forbidden label has a finite alternation number. In particular, any $c \in \gc$ with alternation number $2(d+1)$ induces every $\eta \in 2^{d+1}$.

If $\gc$ is characterized by $\eta$, then in some sense $\eta$ contains all of the information (modulo completeness properties of the order) about how the members of $\gc$ alternate. In particular it determines the alternation number of $\gc$. The converse fails, however; the alternation number is less restrictive, although more robust.

 For instance, the set systems $\{(a,b): a<b \in \QQ\}$ and $\{\{a\}:a \in \QQ\}$ in $\QQ$ both have alternation number 3 with respect to the usual ordering on $\QQ$. But they are clearly quite different.  This difference is reflected in the different $\eta$ which characterize them; these are, respectively, $\langle 1,0,1 \rangle$ and $\langle 1,1 \rangle$.

Table \ref{T:og} expresses a few $\gc \sse 2^X$ and the associated $\eta$ which finitely characterize them. We assume $(X,<)$ is an infinite dense linear order. In each row the forbidden label on the left finitely characterizes the set system on the right. 

\begin{table} 
\begin{center}
\begin{tabular}{|l|l|}
\hline
$\eta$ & $\gc$ \\ \hline
$\langle 0 \rangle$ & $\{X\}$\\
$\langle 1 \rangle$ & $\{\emptyset\}$\\
$\langle 0,0 \rangle$ & $\{X\setminus{a}:a\in X\}$×\\
$\langle 0,1 \rangle$ & $\{(-\infty,a): x \in X\}$\\
$\langle 1,0 \rangle$ & $\{(a,\infty):a \in X\}$\\
$\langle 1,1 \rangle$ & $\{\{a\}:a\in X\}$\\
$\langle 0,0,0 \rangle$ & $\{X \setminus \{a,b\}:a<b \in X\}$\\
$\langle 0,0,1 \rangle$ & $\{ (-\infty,b)\setminus\{a\}:a<b \in X\}$\\
$\langle 0,1,0 \rangle$ & $\{(-\infty,a) \cup (b,\infty): a<b \in X\}$\\
$\langle 1,0,1 \rangle$ & $\{(a,b): a< b \in X\}$\\
$\langle 1,0,1,0 \rangle$ & $\{(a,b) \cup (c,\infty): a< b<c \in X\}$\\
$\langle 1,1,1,0,0,1 \rangle$ & $\{\{a,b\} \cup ((c,e)\setminus \{d\}) : a<b<c<d<e \in X\}$\\ \hline
\end{tabular}
\end{center}
\caption{Some set systems $\gc \sse 2^X$, for $(X,<)$ a linear order, and the forbidden labels $\eta$ that finitely characterize them.}\label{T:og}
 \end{table}

The key for generating Table \ref{T:og} is given in Table \ref{T:key}.  We can view Table \ref{T:key} as a procedure for translating a bit-string into an order-theoretic expression.  Table \ref{T:ex1} illustrates the translation procedure, and Table \ref{T:ex2} shows a reverse translation.

\begin{table}
\begin{center}
\begin{tabular}{|l|l|} \hline
code & translation\\  \hline
$\langle 1 \ldots \rangle$ & do nothing\\
$\langle 0 \ldots \rangle$ & $(-\infty,\ldots $\\
$\langle \ldots 0,0 \ldots \rangle$ & remove point\\
$\langle \ldots 0,1 \ldots \rangle$ & end interval\\
$\langle \ldots 1,0 \ldots \rangle$ & begin interval\\
$\langle \ldots 1,1 \ldots \rangle$ & add point\\
$\langle \ldots 0 \rangle$ & $\ldots \infty)$\\
$\langle \ldots 1 \rangle$ & do nothing\\ \hline
\end{tabular}
\end{center}
\caption{A key for directly translating forbidden labels to set-theoretic expressions.}\label{T:key}
\end{table}

\begin{table} 
\begin{center}
\begin{tabular}{|l|} \hline 
1,1,0,0,1,0,1,0\\ \hline
${}_\emptyset1_{a} 1_{(b}0_{\setminus c}0_{d)}1_{(e}0_{f)}1_{(g}0_{\infty)}$\\ \hline
$\{a\} \cup (b,d)\setminus\{c\} \cup (e,f) \cup (g,\infty)$ \\ \hline
\end{tabular}
\end{center}
\caption{A binary string translated to a set system finitely characterized by it. Spaces between digits can be seen as regions of alternation. We assume $a<b<c<d<e<f<g$.}\label{T:ex1}
 \end{table}

\begin{table}
\begin{center}
\begin{tabular}{|l|} \hline
$(-\infty,b)\setminus\{a\} \cup \{c,d\} \cup (e,f)$ \\ \hline
${}_{(-\infty}0_{\setminus a}0_{b)}1_{c}1_{d}1_{(e}0_{f)}1$\\ \hline
0,0,1,1,1,0,1 \\ \hline
\end{tabular}
\end{center}
\caption{A set system translated to its finitely characterizing forbidden label. We assume $a<b<c<d<e<f$.}\label{T:ex2}
 \end{table}

As can be seen from considering the tables, a forbidden label gives something like the form of a member of $\gc$. Conversely, for any given form of a point-interval system (where the order in which points and intervals occur is held constant) there is an associated forbidden label.   
\section{Order formulas}\label{S:333}

 In this section we show that any maximum family on a sequence of indiscernibles is similar to a family defined on the sequence by a quantifier free formula in the language $L=\{<\}$.

For simplicity we will restrict our attention to dense linear orders without endpoints (DLO), and in particular the structure $(\QQ,<)$. We will make occasional use of the well-known fact that any dense linear order is an $L=\{<\}$ indiscernible sequence.

For any q.f. order formula $\phi(x;y_1,\ldots,y_n)$, define $cof(\phi)$ to be 1 if for some (equivalently any) strictly increasing sequence $a_1<\cdots<a_n<a_{n+1}$ in $\QQ$, $$\QQ \models \phi(a_{n+1};a_1,a_2,\ldots,a_n)$$ and 0 otherwise.

Define $\gc^o_\phi = \{\phi(\QQ;a_1,\ldots,a_n): a_i \in \QQ, a_1 < \cdots < a_n\}$.  This is a subfamily of $\gc_\phi$, corresponding to the sets definable by $\phi$ with parameters in increasing order.

Let $\Sigma$ denote the collection of quantifier free $L=\{<\}$ formulas in at least the variable $x$, partitioned so that $x$ is the only left-hand (non-parameter) variable.

\begin{lemma} \label{C:C1}
For any $d \in \omega$ and $\eta \in 2^{d+1}$, there exists some formula $\psi(x;y_1,\ldots,y_d) \in \Sigma$ such that $\gc^o_\psi$ is finitely characterized by $\eta$.
\end{lemma}
\begin{proof}
  We show this by induction on binary strings.  For the base case, observe that $\langle0\rangle$ finitely characterizes $x=x$ and $\langle 1 \rangle$ finitely characterizes $x \neq x$. We will carry the additional inductive hypothesis that $cof(\phi)=0$ iff $\eta = \mu^\frown\langle 1 \rangle$ for some $\mu$.

For the induction step, suppose for $\phi(x;y_1,\ldots,y_{d}) \in \Sigma$ and $\eta \in 2^{d+1}$, we have $\gc_\phi^o$ finitely characterized by $\eta = \mu^\frown \langle s \rangle$ for $s \in 2$.  

 We must find $\psi_0^s, \psi_1^s \in \Sigma$ such that $\eta {}^\frown \langle i \rangle $ finitely characterizes $\gc_{\psi_i^s}^o$ for $i=0,1$.

We divide the argument into cases based on $cof(\phi)$.  First suppose that $cof(\phi) = 0$, and consequently $s=1$, by inductive hypothesis. Define $$\psi_0^1(x;y_1,\ldots,y_{d+1}) =  \phi(x;y_1,\ldots,y_{d}) \vee x>y_{d+1}$$
\begin{claim}
 $\gc_{\psi_0^1}^o$ is finitely characterized by $\eta {}^\frown \langle 0 \rangle $.
\end{claim}
Let $A \sse \QQ$ be finite, and $C \sse A$.  Suppose there are not $B = b_1<\ldots<b_{d+1}<b_{d+2}$ in $A$ such that $C$ traces $\eta {}^\frown \langle 0 \rangle $ on $B$.  We must show $C \in \gc_{\psi_0^1}^o(A)$. Consider these cases.
\begin{enumerate}
 \item There are $B = b_1<\ldots<b_{d+1}$ in $A$ such that $C$ traces $\eta $ on $B$. 
 \item There are no such $B$.
\end{enumerate}
Suppose Case 2 holds. By inductive hypothesis $C \in \gc_\phi^o(A)$.  Then picking the $y_{d+1}$ parameter sufficiently large, $C \in \gc_{\psi_0^1}^o(A)$.

Suppose Case 1 holds. Let $B = b_1<\ldots<b_{d+1}$ be a least witness, in the sense that $b_{d+1}$ is minimal. Therefore if $A_{<b_{d+1}}:=\{a \in A: a < b_{d+1}\}$ and $C' = C \cap A_{<b_{d+1}}$ then by inductive hypothesis $C' \in \gc_\phi^o(A_{<b_{d+1}})$.  Let this be witnessed by parameters $a_1<\cdots<a_d$.  By indiscernibility, we may assume $a_d < b_{d+1}$.  Since $cof(\phi)=0$, we have $\eta= \mu {}^\frown \langle 1 \rangle $ for some $\mu \in 2^d$.  Then, by the hypothesis on $C$, $A_{\geq b_{d+1}} \sse C$.  Now, picking $a_{d+1}$ between $b_d$ and $b_{d+1}$, we have a parameter set $a_1<\cdots<a_d <a_{d+1}$ putting $C \in \gc_{\psi_0^1}^o(A)$.  

Consider the converse, that if $C \in \gc_{\psi_0^1}^o(A)$, then there are not $B = b_1<\ldots<b_{d+1}<b_{d+2}$ in $A$ such that $C$ traces $\eta {}^\frown \langle 0 \rangle $ on $B$. Suppose, by way of contradiction, that this situation holds.  Let the parameters $a_1<\cdots<a_{d+1}$ witness $C \in \gc_{\psi_0^1}^o(A)$, where $C$ traces $\eta {}^\frown \langle 0 \rangle $ on $B$. Considering the form of $\psi_0^1$, we must have $a_{d+1}>b_{d+2}$, because $b_{d+2} \notin C$. But then $\phi(x;a_1,\ldots,a_d)$ induces $\eta$ on $b_1<\cdots<b_{d+1}$. This gives a contradiction, completing the claim.

Define
$$\psi_1^1(x;y_1,\ldots,y_{d+1}) =  \phi(x;y_1,\ldots,y_{d}) \vee x=y_{d+1}$$
\begin{claim}
 $\gc_{\psi_1^1}^o$ is finitely characterized by $\eta {}^\frown \langle 1 \rangle $.
\end{claim}
The proof of this claim, and the cases for $cof(\phi)=1$, are similar to the above.  Here are the remaining forms, with the proof left to the reader:

$$\psi_0^0(x;y_1,\ldots,y_{d+1}) =  \phi(x;y_1,\ldots,y_{d}) \wedge x\neq y_{d+1}$$
and
$$\psi_1^0(x;y_1,\ldots,y_{d+1}) =  \phi(x;y_1,\ldots,y_{d}) \wedge x<y_{d+1}$$
\end{proof}

Note that the forms of the formulas $\psi_t^s$, for $s,t \in 2$, given in Lemma \ref{C:C1} justify the entries in Table \ref{T:key}.

\begin{definition}
For a given $\eta \in 2^{d+1}$
 $$\Sigma(\eta) := \{\phi(x;y_1,\ldots,y_n) \in \Sigma:\gc^o_\phi(\QQ) \text{ is finitely characterized by } \eta\}$$
\end{definition}

\begin{proposition}\label{P:Pevery}
 For every q.f. order formula $\phi(x;\ybar)$, $\gc_\phi^o$ is finitely characterized by some forbidden label.  In other words, $\{\Sigma(\eta): \eta \in 2^{d+1}, d \in \omega\}$ is a partition of $\Sigma$.
\end{proposition}
\begin{proof}
 The proof is by induction on formulas.  It is easy to see the claim holds for $x=x$ and $x \neq x$.  Now fix some formula $\phi(x;\ybar)$ such that $\gc_\phi^o$ is finitely characterized by $\eta$, where $\eta = \mu^\frown \langle s \rangle$ for some $\mu\in 2^d$ and $s \in 2$. As in Lemma \ref{C:C1}, we carry the inductive hypothesis that $s=1-cof(\phi)$. Let $\bar{\eta}\in 2^{d+1}$ be defined by $\bar{\eta}(i) = 1-\eta(i)$ for all $i < d+1$. 

\begin{claim}
 $\gc_{\neg\phi}^o$ is finitely characterized by $\bar{\eta}$.
\end{claim}
Let $A \sse \QQ$.  Then $c \in \gc_{\neg\phi}^o(A)$ iff  $\QQ \setminus c \in \gc_{\phi}^o(A)$ iff $\QQ \setminus c$ does not induce $\eta$ on $A$ iff $c$ does not induce $\bar{\eta}$ on $A$.  This proves the claim.  

Consider these cases for the remainder of the induction.  All other cases follow from logical manipulations and the claim.  
\begin{enumerate}
 \item $\psi_1(x;\ybar,y) = \phi(x;\ybar) \vee x>y$
\item $\psi_2(x;\ybar,y) = \phi(x;\ybar)  \vee  x<y$
\item $\psi_3(x;\ybar,y) = \phi(x;\ybar)  \vee  x=y$
\end{enumerate}

Consider $\psi_1(x;\ybar,y)$. If $s=0$ then $cof(\phi)=1$ and $\gc_{\psi_1}^o = \gc_\phi^o$.  If $s=1$ then by the arguments from Lemma \ref{C:C1},  $\gc_{\psi_1}^o$ is finitely characterized by $\eta^\frown \langle 0 \rangle$.  
Consider $\psi_2(x;\ybar,y)$. If $s=0$ then $cof(\phi)=1$ and $\gc_{\psi_2}^o = \{\QQ\}$, and $\gc_{\psi_1}^o$ is finitely characterized by $\langle 0 \rangle$. If $s=1$ and $cof(\phi) = 0$ then $\gc_{\psi_2}^o = \gc_{x<y}^o$, and $\gc_{\psi_2}^o$ is finitely characterized by $\langle 0,1 \rangle$. Consider $\psi_3(x;\ybar,y)$. If $s=0$ then $cof(\phi)=1$ and $\gc_{\psi_3}^o =\gc_\phi^o$.  If $s=1$ then by the arguments from Lemma \ref{C:C1}, $\gc_{\psi_3}^o$ is finitely characterized by $\eta^\frown \langle 1 \rangle$.

\end{proof}

We have characterized, up to similarity, the form a maximum formula can take on an indiscernible sequence.  We make this precise in the following corollary.

\begin{corollary} \label{C:C2}
 Suppose $A = \langle \abar_i \rangle_{i \in I}$ is any linearly ordered sequence compatible with $\xbar$ and for some $B \sse M^{|\ybar|}$, and formula $\phi(\xbar;\ybar)$ we have that $\gc_\phi(A)^B$ is $d$-maximum and finitely characterized by the forbidden label $\eta$.

 Let $L'$ consist of a single $2|\xbar|$-ary relation $\prec$ and interpret $(\abar_i \prec \abar_j)^\mcm$ iff $i < j$.  Then there is a quantifier free $L'$ formula $\theta(\xbar;\ybar_1,\ldots,\ybar_d)$ with $|\ybar_i|=|\xbar|$ for $i=1,\ldots,d$ and a set $B'$, such that $\gc_\phi(A)^B \sim \gc_\theta^o(A)^{B'}$.
\end{corollary}
\begin{proof}
 Let $\eta \in 2^{d+1}$ be the characteristizing forbidden label of $\gc_\phi(A)^B$, and let ${\theta}^*(x;y_1,\ldots,y_d)$ be a quantifier free order formula such that ${\theta}^* \in \Sigma(\eta)$, which exists by Lemma \ref{C:C1}.  Define $\theta(\xbar;\ybar_1,\ldots,\ybar_d)$ by replacing each instance of $<$ in ${\theta}^*$ with $\prec$, each instance of $x$ by $\xbar$ and each instance of $y_i$ by $\ybar_i$.

 Let $D = A \cup C$ where $C \sse M^{|\xbar|}$ and $\prec$ is interpreted on $C$ in such a way as to make $(D,\prec) \models $DLO. Define $B' = D^d$.

Then for any finite $A_0 \sse A$, both $\gc_\phi(A_0)^B$ and $\gc_\theta^o(A_0)^{B'}$ are characterized by $\eta$, and so $\gc_\phi(A)^B \sim \gc_\theta^o(A)^{B'}$.

\end{proof}

We now want to show that for any theory $T$, the property of interpreting some $d$-maximum class is equivalent to interpreting $[\omega]^d$.  Define $\ZZ^* = \{(2i,2i+1): i \in \ZZ\}$.  

\begin{lemma} \label{L:L2}
 Let  $\phi(x;\ybar)$ be a quantifier-free $L=\{<\}$ formula such that $\gc_\phi^o$ is finitely characterized by $\eta \in 2^{d+1}$, and define $$\psi_\phi(x_1,x_2;\ybar) = \neg (\phi(x_1;\ybar) \equiv \phi(x_2;\ybar))$$  Then $\gc_{\psi_\phi}^o(\ZZ^*)^{\QQ^{|\ybar|}} = [\ZZ^*]^{\leq d}$.
\end{lemma}
\begin{proof}
 We show this by induction on formulas.  The statement is obvious for the basic formulas.  Suppose the lemma holds for $\phi(x;y_1,\dots,y_n)$, a quantifier-free $L=\{<\}$ formula. By Proposition \ref{P:Pevery}, $\gc_\phi^o$ is finitely characterized by some $\eta \in 2^{d+1}$.  Fix this $\eta$.  We divide the argument into cases depending on $cof(\phi)$.  

Suppose $cof(\phi) = 0$, and consider $\theta(x;y_1,\ldots,y_{n+1}) = \phi(x;\ybar) \vee x > y_{n+1}$.  By the arguments in Lemma \ref{C:C1}, $\gc_\theta^o$ is finitely characterized by $\eta^\frown\langle 0 \rangle$. Define $\psi_\theta(x_1,x_2;\ybar) = \neg (\theta(x_1;\ybar) \equiv \theta(x_2;\ybar))$.

 For any $k\in \omega$, let $B=\{(2i_1,2i_1+1),\ldots,(2i_k,2i_k+1),(2i_{k+1},2i_{k+1}+1)\}$ be given, with $i_1 < \cdots < i_{k+1}$ in $\ZZ$ such that $B\setminus\{(2i_{k+1},2i_{k+1}+1)\}\in \gc_{\psi_\phi}^o(\ZZ^*)$.  We want to show that $B \in \gc_{\psi_\theta}^o(\ZZ^*)$ and $B\setminus\{(2i_{k+1},2i_{k+1}+1)\}\in \gc_{\psi_\theta}^o(\ZZ^*)$.

Let $a_1<\cdots < a_n$ be parameters witnessing that $B\setminus\{(2i_{k+1},2i_{k+1}+1)\}\in \gc_{\psi_\phi}^o(\ZZ^*)$.  By indiscernibility, we may assume $a_n < 2i_{k+1}$.  Then putting $a_{n+1}$ to be the average of $2i_{k+1}$ and $2i_{k+1}+1$ gives a parameter set $a_1<\cdots < a_{n+1}$ witnessing $B \in \gc_{\psi_\theta}^o(\ZZ^*)$.  On the other hand, putting $a_{n+1}$ to be between $2i_{k}+1$ and $2i_{k+1}$ gives a parameter set $a_1<\cdots < a_{n+1}$ witnessing $B\setminus\{(2i_{k+1},2i_{k+1}+1)\}\in \gc_{\psi_\theta}^o(\ZZ^*)$.

It remains to show that there is no $c \in \gc_{\psi_\theta}^o(\ZZ^*)$ with $|c| > d+1$. Suppose there is a sequence of parameters $a_1 < \cdots < a_{n+1}$ such that $\psi_\theta(x_1,x_2;a_1,\ldots,a_{n+1})$ is satisfied by each of $(2i_1,2i_1+1),\ldots,(2i_k,2i_k+1),(2i_{k+1},2i_{k+1}+1)$, with $i_1 < \cdots < i_{k+1}$ for some $k \in \omega$.  Then we must have $a_{n+1}>2i_k+1$, or else $\psi_\phi(2i_{k+1},2i_{k+1}+1;a_1,\ldots,a_{n+1})$ fails.  But then $\psi_\phi(x_1,x_2;a_1,\ldots,a_{n})$ is satisfied by $(2i_1,2i_1+1),\ldots,(2i_k,2i_k+1)$.  Thus $k \leq d$ by inductive hypothesis.

The other cases in the induction are similar and left to the reader. 

\end{proof}

\begin{lemma} \label{L:L3}
 Let $\phi(\xbar;y_1,\ldots,y_n) \in L$, $A\sse M^{|\xbar|}$ and $B \sse M^n$.  Suppose $\gc_{\phi}(A)^{B}$ is infinite and $d$-maximum.  Then there are $A'' \sse M^{|\xbar|}$ and $B' \sse M^n$ with $A''=\{\abar_i:i \in \QQ\}$ and $\eta \in 2^{d+1}$ such that $\gc_{\phi}(A'')^{B'}$ is $d$-maximum and finitely characterized by the forbidden label $\eta$. 
\end{lemma}
\begin{proof}

Let $\prec$ be any linear ordering of $A$, and define a function $h:[A]^{d+1} \rightarrow 2^{d+1}$ which sends each element of $[A]^{d+1}$ to its forbidden label with respect to the ordering $\prec$.  By Ramsey's theorem, there is an infinite homogeneous $A' \sse A$ with respect to $h$.  Note that $\gc_{\phi}(A')^B$ is $d$-maximum.  

We claim that for every finite $A_0 \sse A'$, $\gc_\phi(A_0)^B$ is characterized by $\eta$. Clearly no set in $\gc_\phi(A_0)^B$ induces $\eta$. By Proposition \ref{P:P991}, $\gc_\phi(A_0)^B$ is $d$-maximal, and so for any $c\sse A_0$ not inducing $\eta$ on $A_0$, $c \in \gc_\phi(A_0)^B$. This completes the claim.

Let $C = \{\cbar_i : i \in \QQ\}$ a new set of constants compatible with $\xbar$ and $P(\ybar)$ a new predicate. For every finite subset $C_0 = \{\cbar_{i_1},\ldots,\cbar_{i_n}\}$ of $C$, with $i_1 < \cdots < i_n$, let $\rho(C_0)$ express that $\gc_\phi(C_0)^{P(\ybar)}$ is characterized by $\eta$. 
The set of sentences $\{\rho(C_0):C_0\sse C,\text{ finite}\}$ is easily seen to be consistent.  Let this be witnessed by a model $\mcn$.  Now $A'' = C^\mcn$ and $B' = P^\mcn$ are as desired.  Since $\mcn$ embeds into the monster model $\mcm$, we are done.

\end{proof}

Recall that we use $dpR(n)$ for $n \in \omega$ to refer to the maximum depth of an ICT pattern in $n$ variables.

\begin{theorem} \label{T:T2}
 For any theory $T$ and $n,d \in \omega$ the following are equivalent.
\begin{enumerate}
 \item $T$ interprets an infinite $d$-maximum VC family in $n$ parameters.
 \item $T$ interprets $[\omega]^d$ in $n$ parameters.
 \item dpR($n) \geq d$.
\end{enumerate}
\begin{proof}
 The direction (2) $\rightarrow$ (1) is clear, since $[\omega]^d$ is an infinite $d$-maximum VC family.  

(1) $\rightarrow$ (2):  

Let $\gc_{\phi}(A)^B$ with $\phi(\xbar;y_1,\ldots,y_n)$, $A\sse M^{|\xbar|}$ and $B \sse M^n$ constitute an interpretation of some  infinite $d$-maximum family.  By Lemma \ref{L:L3}, we may assume that $A=\{\abar_i:i\in\QQ\}$ and that $\gc_{\phi}(A)^B$ is characterized by a forbidden label $\eta \in 2^{d+1}$.

  Let $\theta(\xbar;\ybar_1,\ldots,\ybar_d)$ be an $L'$ formula as in the statement of Corollary \ref{C:C2}.  Then $\gc_\phi(A)^B \sim \gc_\theta^o(A)^{A^d}$, since for any finite $A_0 \sse A$, both $\gc_\phi(A_0)^B$ and $\gc_\theta^o(A_0)^{A^d}$ are characterized by $\eta$.

Define $A' = \langle (\abar_{2i},\abar_{2i+1}): i \in \ZZ\rangle$.  Define $\psi_\phi(\xbar_1,\xbar_2;y_1,\ldots,y_n) = \neg(\phi(\xbar_1;y_1,\ldots,y_n) \equiv \phi(\xbar_2;y_1,\ldots,y_n))$, and let $\psi_\theta(\xbar_1,\xbar_2;\ybar_1,\ldots,\ybar_d) = \neg(\theta(\xbar_1;\ybar_1,\ldots,\ybar_d) \equiv \theta(\xbar_2;\ybar_1,\ldots,\ybar_d))$.  Then $\gc_{\psi_\phi}(A')^B \sim \gc_{\psi_\theta}^o(A')^{A^d}$.  But by Lemma \ref{L:L2}, $\gc_{\psi_\theta}^o(A')^{A^d} = [A']^{\leq d}$, and therefore  $\gc_{\psi_\phi}(A')^B \sim [A']^{\leq d}$.

By compactness, we can find a countably infinite $A''$ and a set $B' \sse M^n$ such that $\gc_{\psi_\phi}(A'')^{B'} = [A'']^d$.  Thus we have an interpretation of $[\omega]^d$ in $n$ parameters.

(2) $\rightarrow$ (3):  Suppose there is a formula $\phi(\xbar;\ybar)$ with $|\ybar| =n$ and infinite sets $A\sse M^{|\xbar|}$, $B \sse M^n$ such that $\gc_\phi(A)^B = [A]^d$.  Let $\Gamma$ be a set of sentences 
 expressing that $\{\psi_1(\ybar;\xbar_1),\ldots,\psi_d(\ybar;\xbar_d)\}$ witnesses a depth $d$ ICT pattern with $\psi_i(\ybar;\xbar_i) = \phi(\xbar_i;\ybar)$ for $i=1,\ldots,d$.  Then by compactness and choice of $\phi$, $\Gamma$ is consistent and consequently dpR($n) \geq d$.

(3) $\rightarrow$ (2):  Suppose the tuples $\langle \bbar_{i,j}: i \leq d, j < \omega\rangle$ and the formulas $\psi_1(\xbar;\ybar_1),\ldots,\psi_d(\xbar;\ybar_d)$ constitute a depth $d$ ICT pattern in $T$ with $|\bbar_{i,j}| = |\ybar_i|$ and $|\xbar| = n$.  Define $$\phi(\xbar;\ybar_1,\ldots,\ybar_d) = \neg( \psi_1(\xbar;\ybar_1) \equiv \cdots \equiv \psi_d(\xbar;\ybar_d))$$

Let $A = \{\bbar_{1,j}{}^{\frown}\cdots{}^{\frown}\bbar_{d,j} : j < \omega\}$.  Then with $\phi^*(\ybar_1,\ldots,\ybar_d;\xbar) = \phi(\xbar;\ybar_1,\ldots,\ybar_d)$, there is clearly a set $B \sse M^n$ such that $\gc_{\phi^*}(A)^B = [A]^d$.  Thus we have an interpretation of $[\omega]^d$ in $n$ parameters.

\end{proof}

\end{theorem}

Naturally any infinite set can be substituted for $\omega$ in Theorem \ref{T:T2}.

\section{Relations to other notions}
 In this section we relate Theorem \ref{T:T2} to some results of others.  

\begin{definition}
For a formula $\phi(\xbar;\ybar)$ let $max(\phi)$ be defined as the maximum $d\in \omega$, should it exist, for which $\gc_\phi(A)^B$ is $d$-maximum, for some infinite $A \sse M^{|\xbar|}$ and $B \sse M^{|\ybar|}$.  If no such $d$ exists, put $max(\phi) = \infty$.  For $n \in \omega$ let $max(n) = sup\{max(\phi(\xbar;\ybar)):|\ybar|=n\}$.
\end{definition}

We may summarize Theorem \ref{T:T2} by the statement $max(n)=dpR(n)$ for all $n \in \omega$.  

\begin{lemma}[Theorem 2.7 of \cite{ItOnUs11}]
 If $dpR(1) \leq n$ then $dpR(k) \leq kn$ for all $k \in \omega$.
\end{lemma}

\begin{corollary}
 If $max(1) \leq n$ then $max(k) \leq kn$ for all $k \in \omega$.
\end{corollary}

In particular, for any dp-minimal theory and any $n \in \omega$, $max(n) = n$.  

If $VC_{ind}$-density is defined as in \cite{GuHi11}, and $\phi^*(\ybar;\xbar) = \phi(\xbar;\ybar)$ is the dual formula, then for any $\phi(\xbar;\ybar)$, $max(\phi^*) \leq VC_{ind}$-density of $\phi$.  This can be seen by using Lemma \ref{L:L3} and applying Ramsey's Theorem.  It seems plausible that the converse may hold as well, though this would appear to require some work.  

The following is an easy variation of Theorem 3.14 of \cite{Gu11}.

\begin{theorem}
 If $\phi(\xbar;\ybar)$ has $max(\phi^*) = 1$, then $\phi$ has UDTFS.
\end{theorem}

Whether the corresponding statement holds for $max(\phi^*) = 2$ is an interesting open question.

\bibliographystyle{amsplain}
\small\bibliography{refs.bib}

\providecommand{\bysame}{\leavevmode\hbox to3em{\hrulefill}\thinspace}
\providecommand{\MR}{\relax\ifhmode\unskip\space\fi MR }
\providecommand{\MRhref}[2]{%
  \href{http://www.ams.org/mathscinet-getitem?mr=#1}{#2}
}
\providecommand{\href}[2]{#2}
\begin{thebibliography}{10}

\bibitem{Ad10}
Hans Adler, \emph{Introduction to theories without the independence property},
  Archive for Mathematical Logic, to appear.

\bibitem{Ad07}
\bysame, \emph{Strong theories, burden, and weight}, 2007.

\bibitem{AsDoHaMaSt11}
Matthias Aschenbrenner, Alf Dolich, Deirdre Haskell, H.~Dugald MacPherson, and
  Sergei Starchenko, \emph{{V}apnik-{C}hervonenkis density in some theories
  without the independence property, i}, preprint arxiv: 1109.5438.

\bibitem{Du99}
R.M. Dudley, \emph{Uniform central limit theorems}, Cambridge University Press,
  New York, 1999.

\bibitem{Fl89}
S.~Floyd, \emph{Space-bounded learning and the {V}apnik-{C}hervonenkis
  dimension}, Ph.D. thesis, U.C. Berkeley, 1989.

\bibitem{Gu11}
Vincent Guingona, \emph{On uniform definability of types over finite sets}, J.
  Symbolic Logic, to appear.

\bibitem{GuHi11}
Vincent Guingona and Cameron Donnay~Hill, \emph{Local dp-rank and {VC}-density
  over indiscernible sequences}, preprint arXiv:1108.2554.

\bibitem{ItOnUs11}
Itay Kaplan, Alf Onshuus, and Alexander Usvyatsov, \emph{Additivity of the
  dp-rank}, preprint arxiv: 1109.1601, 2010.

\bibitem{L92}
Laskowski, \emph{Vapnik-{C}hervonenkis classes of definable sets}, J. London
  Math. Soc. \textbf{45} (1992), no.~2, 377--384.

\bibitem{OnUs11}
Alf Onshuus and Alex Usvyatsov, \emph{On dp-minimality, strong dependence, and
  weight}, The J. of Symb. Logic \textbf{76} (2011), no.~3, 737--758.

\bibitem{Sa72}
N.~Sauer, \emph{On the density of families of sets}, Journal of Combinatorial
  Theory \textbf{13} (1972), 145--147.

\bibitem{Sh07}
Saharon Shelah, \emph{Strongly dependent theories}, Israel Journal of
  Mathematics, accepted.

\bibitem{Sh72}
\bysame, \emph{A combinatorial problem: stability and order for models and
  theories in infinitary languages}, Pacific Journal of Mathematics \textbf{41}
  (1972), no.~1, 247--261.

\bibitem{Sh06}
\bysame, \emph{Dependent first order theories, continued}, Israel Journal of
  Mathematics \textbf{173} (2009), no.~1, 1--60.

\bibitem{VaCh71}
V.~Vapnik and A.~Chervonenkis, \emph{On the uniform convergence of relative
  frequencies of events to their probabilities}, Theory of Probability and its
  Applications \textbf{16} (1971), no.~2, 264--280.

\bibitem{We87}
E.~Welzl, \emph{Complete range spaces}, Unpublished notes, 1987.

\end{thebibliography}

\end{document}